\documentclass[preprint,12pt]{elsarticle}

\usepackage{}
\usepackage{mathrsfs}

\usepackage{amssymb}
\usepackage{amsmath}
\usepackage{stmaryrd}

\biboptions{compress}
\newtheorem{theorem}{Theorem}[section]
\newtheorem{lemma}[theorem]{Lemma}
\newtheorem{corollary}[theorem]{Corollary}

\newdefinition{definition}[theorem]{Definition}
\newdefinition{remark}[theorem]{Remark}
\newdefinition{example}[theorem]{Example}
\newdefinition{proposition}[theorem]{Proposition}
\newproof{proof}{Proof}

\journal{}

\begin{document}

\begin{frontmatter}



\title{Ekeland's Variational Principle for An $\bar{L}^{0}-$Valued Function on A Complete Random Metric Space\tnoteref{t1}}
\tnotetext[t1]{Supported by NNSF No. 10871016}


\author[se]{Tiexin Guo\corref{cor1}}
\ead{txguo@buaa.edu.cn} \cortext[cor1]{Corresponding author}
\address[se]{LMIB and School of Mathematics and Systems Science, Beihang University, Beijing 100191, P.R. China}

\author[gs]{Yujie Yang}
\ead{yangyujie007@163.com}
\address[gs]{LMIB and School of Mathematics and Systems Science, Beihang University, Beijing 100191, P.R. China}

\begin{abstract}
 Motivated by the recent work on conditional risk measures, this paper  studies the  Ekeland's variational
  principle for a  proper, lower semicontinuous and lower bounded $\bar{L}^{0}-$valued function, where $\bar{L}^{0}$ is the
  set of equivalence classes of extended real-valued random variables on a probability space. First, we
  prove a general form of  Ekeland's variational principle for such a function defined on a complete random
  metric space. Then, we give a more precise  form of Ekeland's variational principle for such  a local function on a
  complete random normed module. Finally,  as applications, we establish  the Bishop-Phelps theorem in a complete random normed
  module under the framework of random conjugate spaces.
\end{abstract}


\end{frontmatter}

\section{Introduction }
\label{}

In 1960,  Bishop and  Phelps \cite{B-R} showed that a nonempty
closed convex subset  of a Banach space  admits ``many'' support
points and support functionals (``many'' means `` norm dense in the
appropriate set''), in particular presented a key ordering
technique. In 1972, Ekeland \cite{Ekeland-1972} extended the
ordering technique to a complete metric space so that he could
establish the famous variational principle for a  proper lower
semicontinuous and lower bounded extended real-valued function
defined on a complete metric space together with a series of
applications to many fields from control theory to global analysis
\cite{Ekeland-1974, Ekeland-1979}. Br$\phi$ndsted
\cite{Brndsted-1974} and Brezis and Browder \cite{Brezis}
generalized the ordering technique of  Bishop and Phelps to a
general   ordering principle in nonlinear analysis and gave some
further applications. Subsequently, the Ekeland's variational
principle was proved to be  equivalent to  the famous Caristi's
fixed point theorem \cite{Caristi}, the drop theorem and the petal
theorem \cite{penot} and the completeness of a metric space
\cite{Sullivan}. To meet the needs of  the vectorial optimization, a
lot of scholars have generalized the Ekeland's variational principle
from real-valued functions to  vector-valued (namely, partially
ordered vector space-valued) functions
 since 2000, see  \cite{Gopfert-Tammer-Za, Finet, Yousuke Araya} and their references for details.

Motivated by the recent applications of random metric theory to conditional risk measures\cite{Filipovic-Kupper,Guotx-relation,Guotx-recent}, in this paper we study the  Ekeland's variational
principle  for a function from a random metric space (briefly, an $RM$ space) with  base $(\Omega,\mathcal {F},P)$(a probability space) to $\bar{L}^{0}(\mathcal {F})$, where $\bar{L}^{0}(\mathcal {F})$ denotes the set of equivalence classes of extended real-valued random variables defined on $(\Omega,\mathcal {F},P)$.

 An $RM$ space is a random generalization of  an ordinary metric space,
whose original definition was given in \cite{SS}. According to \cite{SS}, the random distance $d(p,q)$ between two points $p$ and $q$ in an $RM$ space $(E,d)$ with  base $(\Omega,\mathcal {F},P)$ is a nonnegative random variable defined on $(\Omega,\mathcal {F},P)$. A new version of the  original $RM$ space was presented in \cite{Guotx-basic} in the course of the development of random metric theory in the direction of functional analysis, according to which the random distance $d(p,q)$ between two points $p$ and $q$ in an $RM$ space $(E,d)$ with  base $(\Omega,\mathcal {F},P)$ is the equivalence class of a nonnegative random variable defined on $(\Omega,\mathcal {F},P)$. The study of an $RM$ space is different from that of an ordinary metric space in that the distance function on  an ordinary metric space induces a unique uniformity, whereas the random distance function on an $RM$ space $(E,d)$ with  base $(\Omega,\mathcal {F},P)$ can induce two kinds of uniformities, namely the $d_{\varepsilon,\lambda}$-uniformity and the $d_{c}$-uniformity, which are defined as follows, respectively. Let $(E,d)$ be an $RM$ space  with  base $(\Omega,\mathcal {F},P)$, $\varepsilon$ and $\lambda$ two positive real numbers such that $0<\lambda<1$,  $U(\varepsilon,\lambda)=\{(p,q)\in E\times E
:P\{\omega\in\Omega:d(p,q)(\omega)<\varepsilon\}>1-\lambda\}$ and
$U(d)=\{ U(\varepsilon,\lambda):\varepsilon>0,0<\lambda<1\}$, then
$U(d)$ forms a base for some metrizable  uniformity  on $E$,
 called
 the $d_{\varepsilon,\lambda}$-uniformity, whose topology is called the
$(\varepsilon,\lambda)$-topology determined by the random metric
$d$, denoted by $\mathcal {T}_{\varepsilon,\lambda}$. It is known from \cite{Dunford} that $\bar{L}^{0}(\mathcal {F})$  is a complete lattice under the ordering $\leqslant$:
$\xi\leqslant \eta$ iff $\xi^{0}(\omega)\leqslant\eta^{0}(\omega)$,
for almost all $\omega$ in $\Omega$ (briefly, a.s.), where $\xi^{0}$
and $\eta^{0}$ are arbitrarily chosen representatives of $\xi$ and
$\eta$, respectively. Furthermore, every subset A of
$\bar{L}^{0}(\mathcal {F})$ has a supremum and infimum, denoted by
$\vee A$ and $\bigwedge A$, respectively. It is clear that  $ L^{0}(\mathcal
{F})$ consisting of all  equivalence classes of real-valued random variables defined on $(\Omega,\mathcal {F},P)$, as a sublattice of $\bar{L}^{0}(\mathcal {F})$, is also a
complete lattice in the sense that every subset with an upper bound has
a supremum. Let $L^{0}_{+}(\mathcal {F})=\{\xi\in L^{0}({\cal
F})\,|\,\xi\geqslant0\}$ and $L^{0}_{++}(\mathcal {F})=\{\xi\in L^{0}({\cal
F})\,|\,\xi>0 ~on~ \Omega\}$, where $\xi>\eta ~on~ \Omega$
means $\xi^{0}(\omega)>\eta^{0}(\omega)$ for $P-$almost all $\omega \in \Omega$ (briefly, a.s.) for any $\xi$ and $\eta$ in $\bar{L}^{0}({\cal F})$,
where $\xi^{0}$ and $\eta^{0}$ are arbitrarily  chosen
representatives of $\xi$ and $\eta$ in $\bar{L}^{0}({\cal F})$,
respectively. Let $(E,d)$ be an $RM$ space with base $(\Omega,{\cal F},P)$, for  any
$\varepsilon \in L^{0}_{++}(\mathcal {F})$, let $U(\varepsilon )=\{(p,q)\in
E\times E :d(p,q)\leq \varepsilon \}$ and
$\tilde{U}(d)=\{U(\varepsilon ):\varepsilon \in L^{0}_{++}(\mathcal {F})\}$, then
$\tilde{U}(d)$ forms a base for some Hausdorff uniformity
on $E$, called the $d_{c}$-uniformity, whose topology is called
the ${\cal T}_{c}$-topology determined by the random metric $d$, denoted by ${\cal T}_{c}$.

Let $(E,d)$ be an $RM $ space  with base $(\Omega,{\cal F},P)$, a function $f:E\rightarrow \bar{L}^{0}(\mathcal {F})$ is proper if $f(x)>-\infty$ on $\Omega$ for every
$x\in E$ and $dom(f):=\{x\in E~|~ f(x)<+\infty \textmd{ on }  \Omega \}$, denoting its effective domain, is not empty; $f$ is
bounded from below if there exists $\xi \in L^{0}(\mathcal {F})$ such that $f(x)\geq \xi$ for any $x \in E$ and $f$ is  ${\cal
T}_{\varepsilon,\lambda}$( resp.,${\cal
T}_{c}$)-lower semicontinuous if its epigraph $epi(f):=\{(x,r) \in E\times L^{0}(\mathcal {F})| f(x)\leq r\}$ is closed in $(E,\mathcal{
T}_{\varepsilon,\lambda})\times (L^{0}(\mathcal {F}),\mathcal{
T}_{\varepsilon,\lambda})$ (accordingly, $(E,\mathcal{ T}_{c})\times  (L^{0}(\mathcal {F}),\mathcal{ T}_{c})$), where $L^{0}(\mathcal
{F})$ forms an $RM$ space endowed with the random metric $d:
L^{0}(\mathcal {F}) \times  L^{0}(\mathcal {F}) \rightarrow L^{0}_{+}(\mathcal {F})$  by
$d(p,q)=|p-q| $ for any $p$ and $q \in L^{0}(\mathcal {F})$.

It is clear that the $\mathcal {T}_{c}$-topology is much stronger than the $(\varepsilon,\lambda)$-topology and that the $(\varepsilon,\lambda)$-topology is quite natural from the viewpoint of probability theory, for example the $(\varepsilon,\lambda)$-topology  on $L^{0}(\mathcal {F})$ is exactly the one of convergence in probability $P$. In Section 2 of this paper,  making full use of the advantage of the $(\varepsilon,\lambda)$-topology we first establish the Ekeland's variational principle and equivalent Caristi's fixed point theorem for a function $f$ from a $d_{\varepsilon,\lambda}$-complete $RM$ space  with base $(\Omega,{\cal F},P)$ to $\bar{L}^{0}(\mathcal {F})$, which is proper, ${\cal
T}_{\varepsilon,\lambda}$-lower semicontinuous and bounded from below. Since our definition of lower semicontinuity is weaker and more natural than that given in the earlier two approaches \cite{You-Zhu,Bai-Xiong}, our results also improve those in \cite{You-Zhu,Bai-Xiong}.

Since $\bar{L}^{0}(\mathcal {F})$ is a partially ordered set and an $RM$ space does not possess the rich stratification structure, for a
function $f$ from a $d_{\varepsilon,\lambda}$-complete $RM$ space  with base $(\Omega,{\cal F},P)$ to $\bar{L}^{0}(\mathcal {F})$, which
 is proper and bounded from below, there is unnecessarily an element  $x_{\varepsilon} \in E$ for any $\varepsilon \in L^{0}_{++}(\mathcal {F})$
 such that $f(x_{\varepsilon})\leq \bigwedge f(E)+\varepsilon$  so that we can not give the location of the approximate minimal point in
 the Ekeland's variational principle given in Section 2 of this paper. But when we come to a special class of $RM$ spaces---random normed
 modules(briefly, $RN$ modules) first introduced in \cite{Guotx-basic,exte}, which are a stronger
  random generalization of ordinary normed spaces, whose $(\varepsilon,\lambda)$-topology is exactly the
  frequently used $(\varepsilon,\lambda)$-linear topology \cite{Guotx-relation} and whose ${\cal T}_{c}$-topology is just
  the locally $L^{0}-$convex topology first introduced in \cite{Filipovic-Kupper},  since $RN$ modules possess the rich stratification
  structure, let $(E,\|\cdot\|)$ be an $RN$ module over the real number field $R$ with base
$(\Omega,{\cal F},P)$ such that $E$ has the countable concatenation property first introduced
 in \cite{Guotx-relation} and $f:E\rightarrow \bar{L}^{0}(\mathcal {F})$ a function with the local property
 such that $f$ is proper and bounded from below,  we can prove that there is an element  $x_{\varepsilon} \in E$  for any
  $\varepsilon \in L^{0}_{++}(\mathcal {F})$ such that $f(x_{\varepsilon})\leq \bigwedge f(E)+\varepsilon$, so that we
  can give a more precise form of Ekeland's variational principle for such a ${\cal T}_{\varepsilon,\lambda}$-lower semicontinuous $f$
  in Section 3 of this paper. Further, based on the relations between the basic results derived from the two kinds of
  topologies \cite{Guotx-relation}, we can also establish the precise form of Ekeland's variational principle for
  such a ${\cal T}_{c}$-lower semicontinuous $f$. Since Guo in Section 8 of \cite{Guotx-recent} has proved that the complete $RN$ module $L^{p}_{\mathcal {F}}(\mathcal {E})$ constructed in \cite{Filipovic-Kupper} is an universally suitable
  model space for  a conditional risk measure, the so-called conditional risk measure is exactly a proper, $L^{0}(\mathcal {F})$-convex,
  cash invariant and monotone function from $L^{p}_{\mathcal {F}}(\mathcal {E})$ to $\bar{L}^{0}(\mathcal {F})$, which together with a
  general $L^{0}(\mathcal {F})$-convex function from an $RN$ module to $\bar{L}^{0}(\mathcal {F})$ has the local property, and so our results also cover the optimization problems for such  functions, which will be discussed in more details in a forthcoming paper.

Since the theory of random conjugate spaces has played an essential
role in both the development of the theory of $RN$ modules and their
applications to conditional risk measures, our previous works have
been focused on the theory of random conjugate spaces of $RN$
modules \cite{ Guotx-relation, Guotx-recent, Guotx-James,
Guotx-Alaoglu, Guotx-Xiao-Chen, Guotx-shiguang}. This paper
continues the study of the theory of random conjugate spaces,
precisely speaking, as applications of the results in Section 3,
Section 4 is devoted to establishing  the Bishop-Phelps theorem in
complete $RN$ modules under the framework of random conjugate spaces
and under the two kinds of topologies.

\section{The Ekeland's variational principle on a $d_{\varepsilon,\lambda}$-complete $RM$ space}
 \label{}

 Throughout this paper,
$(\Omega,{\cal F},P)$ denotes a probability space, $K$ the real
number field $R$
 or the complex number field $C$, $N$ the set of positive integers,   $\bar{L}^{0}(\mathcal {F})$
the set of equivalence classes of extended real-valued  random variables on $\Omega$ and $L^{0}({\cal F},K)$
the algebra of equivalence classes of $K-$valued ${\cal
F}-$measurable random variables on $\Omega$ under the ordinary
scalar multiplication, addition and multiplication operations on
equivalence classes, denoted by $L^{0}(\mathcal {F})$ when $K=R$.

The pleasant properties of the complete lattice $\bar{L}^{0}(\mathcal {F})$ (see the introduction for the notation $\bar{L}^{0}(\mathcal {F})$ )
are summarized as follows:

\begin{proposition}[\cite{Dunford}]
For every subset $A$ of
$\bar{L}^{0}(\mathcal {F})$,  there exist countable subsets
$\{a_{n}\,|\,n\in N\}$ and  ${\{b_{n}\,|\,n\in N}\}$ of $A$ such
that $\vee_{n\geqslant1}$ $a_{n}=\vee A$ and $\wedge_{n\geqslant1}$
$b_{n}=\wedge A$. Further, if $A$ is directed (dually directed) with
respect to $\leq$, then the above $\{a_{n}\,|\,n\in N\}$
(accordingly,  $\{b_{n}\,|\,n\in N\}$) can be chosen as
nondecreasing (correspondingly, nonincreasing) with respect to
$\leq$.
\end{proposition}

Specially, $L^{0}_{+}(\mathcal {F})=\{\xi\in L^{0}({\cal F})\,|\,\xi\geqslant0\}$, $L^{0}_{++}(\mathcal {F})=\{\xi\in L^{0}({\cal F})\,|\,\xi>0 ~on~ \Omega\}$.

 As usual, $\xi>\eta$ means
$\xi\geqslant\eta$ and $\xi\neq\eta$, whereas $\xi>\eta ~on~ A$
means $\xi^{0}(\omega)>\eta^{0}(\omega)$ a.s. on A for any $A \in
\mathcal {F}$ and  $\xi$ and $\eta$ in $\bar{L}^{0}({\cal F})$,
where $\xi^{0}$ and $\eta^{0}$ are arbitrarily  chosen
representatives of $\xi$ and $\eta$,
respectively.

For any $A\in \mathcal {F}$, $A^{c}$ denotes the complement of $A$,
$\tilde{A}=\{ B \in \mathcal {F}| P(A\Delta B)=0\} $ denotes the
equivalence class of $A$, where $\Delta$ is the symmetric difference
operation, $I_{A}$ the characteristic function of $A$, and
$\tilde{I}_{A}$ is used to denote the equivalence class of $I_{A}$;
given two $\xi$ and $\eta$ in $\bar{L}^{0}({\cal F})$, and $A=\{\omega \in
\Omega:\xi^{0}\neq \eta^{0}\}$, where $\xi^{0}$ and $\eta^{0}$ are
arbitrarily chosen representatives of $\xi$ and $\eta$ respectively,
then we always write $[\xi \neq \eta]$ for the equivalence class of
$A$ and $I_{[\xi \neq \eta]}$ for $\tilde{I}_{A}$, one can also
understand the implication of such notations as $I_{[\xi \leq
\eta]}$, $I_{[\xi < \eta]}$ and $I_{[\xi = \eta]}$.

 For an arbitrary chosen representative $\xi^{0}$ of
$\xi\in L^{0}({\cal F},K)$, define the two $\mathcal {F}-$measurable
random variables $(\xi^{0})^{-1}$ and $|\xi^{0}|$ by
$(\xi^{0})^{-1}(\omega)=\frac{1}{\xi^{0}(\omega)}$ if $\xi^{0}(\omega)\neq 0 $, and
$(\xi^{0})^{-1}(\omega)=0$ otherwise, and by
$|\xi^{0}|(\omega)=|\xi^{0}(\omega)|,\forall \omega\in \Omega$. Then the
equivalence class $\xi^{-1}$ of $(\xi^{0})^{-1}$ is called the
generalized inverse of $\xi$ and the equivalence class $|\xi|$ of
$|\xi^{0}|$ is called the absolute value of $\xi$. It is clear that
$\xi\cdot \xi^{-1}=I_{[\xi\neq 0]}$.



The main result in this section  is the  Ekeland's variational
principle for a proper lower semicontinuous and lower bounded
$\bar{L}^{0}-$valued function on a $d_{\varepsilon,\lambda}$-complete random metric space, namely Theorem 2.10
below. To prove this theorem, we first give some preliminaries in
Section 2.1.

\subsection{A general principle on ordered sets }
\label{}

\begin{definition} [\cite{Guotx-basic}]
An ordered pair $(S,d)$ is
called {\it{a random metric space (briefly, an $RM$ space) with base $(\Omega,{\cal F},P)$}}
if $S$ is a nonempty set and the mapping $d$ from $S\times S$ to
$L^{0}_{+}(\mathcal {F})$ satisfies the following three axioms:\\
\noindent($RM$-1) $d(p,q)=0$ $\Leftrightarrow$ $p=q$;\\
\noindent($RM$-2) $d(p,q)=d(q,p), \forall p,q \in S$;\\
\noindent($RM$-3) $d(p,r)\leq d(p,q)+d(q,r),\forall p,q,r \in S$,\\
 where $d(p,q)$ is called the random distance between $p$
and $q$.
\end{definition}

\begin{example}
 Clearly, $(L^{0}(\mathcal {F}),d)$ is an $RM$
space  with base $(\Omega,{\cal F},P)$, where the mapping $d:
L^{0}(\mathcal {F}) \times  L^{0}(\mathcal {F})       \rightarrow L^{0}_{+}(\mathcal {F})$ is defined by
$d(p,q)=|p-q| $ for any $p$ and $ q \in L^{0}(\mathcal {F})$.
\end{example}

 Let $(E,d)$ be an $RM$ space  with base
$(\Omega,{\cal F},P)$, define $\mathcal {V}:E\times E\rightarrow
D^{+}$ by $\mathcal {V}_{p,q}(t)=P\{\omega \in \Omega:
d(p,q)(\omega)< t\}$ for all nonnegative numbers $t$ and $p$ and $q$
in $E$, where $D^{+}=\{ F:[0,+\infty)\rightarrow [0,1]| F \textmd{
is nondecreasing,  left continuous }$ $\textmd{on
}(0,+\infty),F(0)=0 \textmd{ and } \lim_{t\rightarrow +\infty}
F(t)=1\} $, then $(E,\mathcal {V})$ is a Menger probabilistic metric
space under the $t$-norm $W:[0,1]\times [0,1]\rightarrow [0,1]$
defined by $W(a,b)=max(a+b-1,0)$ for all $a$ and $b$ in $[0,1]$, and
the $d_{\varepsilon,\lambda}$-uniformity and its
$(\varepsilon,\lambda)$-topology $\mathcal
{T}_{\varepsilon,\lambda}$ on $E$ (see the introduction of this
paper) are those induced from the probabilistic metric $\mathcal
{V}$  \cite{SS}. It is clear that  a sequence $\{p_{n},n \in N\}$
converges in the $(\varepsilon,\lambda)$-topology to some point $p$
in $(E,d)$  iff $\{d(p_{n},p),n\in N\}$ converges in
probability $P$ to 0, in particular, the
$(\varepsilon,\lambda)$-topology  on $L^{0}(\mathcal {F})$ is
exactly the one of convergence in probability $P$.

The $d_{c}$-uniformity on an $RM$ space $(E,d)$ (see the introduction of this paper) is peculiar to the random distance $d$, the idea of our introducing the $d_{c}$-uniformity and its topology $\mathcal {T}_{c}$ is motivated by the work of D. Filipovi\'{c}, et al's introducing the locally $L^{0}$-convex topology for $RN$ modules \cite{Filipovic-Kupper}.

We say that an $RM$ space $(E,d)$ is  $d_{\varepsilon,\lambda}$-complete (resp., $d_{c}$-complete)  if the
 $d_{\varepsilon,\lambda}$-uniformity (accordingly, $d_{c}$-uniformity) is complete. From now on, the $(\varepsilon,\lambda)$-topology and the $\mathcal {T}_{c}$-topology induced  by the $d_{\varepsilon,\lambda}$-uniformity and $d_{c}$-uniformity for every $RM$ space $(E,d)$ are denoted by $\mathcal{T}_{\varepsilon,\lambda}$ and $\mathcal {T}_{c}$, respectively, whenever no confusion occurs.

\begin{definition}
Let $X$ be a Hausdorff space and $f:X\rightarrow
\bar{L}^{0}(\mathcal {F})$, then\\

\noindent(1) $dom(f):=\{x\in X~|~ f(x)<+\infty \textmd{ on } \Omega\}$ is called
the effective domain of $f$.\\
\noindent(2) $f$ is  proper if $f(x)>-\infty$ on $\Omega$ for every $x\in X$
and $dom(f)\neq \emptyset$.\\
\noindent(3)$f$ is bounded from below (resp., bounded from above) if there
exists $\xi \in L^{0}(\mathcal {F})$ such that $f(x)\geq \xi$
(accordingly, $f(x)\leq \xi$ ) for any $x \in X$.

\end{definition}

We first give the following:

\begin{lemma}
 Let $(X,\mathscr{U})$ be a complete Hausdorff
uniform space,  $\leq$  a partial ordering on $X$ and $\phi:
X\rightarrow \bar{L}^{0}(\mathcal{F})$  proper and  bounded from below.  Further, if $x\leq y$ $\Rightarrow $
$\phi(y)\leq \phi(x)$,  then for each totally ordered subset $M$ in $X$
such that $\phi(M)\subset L^{0}(\mathcal{F})$, $\{\phi(m), m\in M\}$ is a
$d_{\varepsilon,\lambda}-$Cauchy net  in
$L^{0}(\mathcal{F})$.
\end{lemma}

\begin{proof}
 Since a totally ordered set is also a
directed set, $\{\phi(m), m\in M\}$ can be naturally understood as a net defined on
$M$. By the hypothesis that $\phi$ is  bounded from below, then
 $\phi(M)$ has a infimum $\bigwedge \{\phi(x):x\in M\}:=\eta $ by the completeness of the lattice $L^{0}(\mathcal{F})$.
 Since $x\leq y$ $\Rightarrow $ $\phi(y)\leq
\phi(x)$ and $M$ is a totally ordered subset, it follows that $\{\phi(m): m\in M\}$ is  directed downwards. From Proposition 2.1,  there exists a sequence
$\{\phi(x_{n}):n \in N\} \subset \phi(M)$ such
that $\{\phi(x_{n}):n \in N \}$
 converges  to $\eta$ in a nonincreasing way, and hence also converges to $\eta$ in probability $P$, then $\{\phi(x_{n}):n \in N \}$ is, of
 course, a Cauchy sequence in  probability $P$. Since $\{\phi(m), m\in M\}$ is a nonincreasing net with respect to $\leq$ on $M$, then it must be a $d_{\varepsilon,\lambda}-$Cauchy net in
$L^{0}(\mathcal{F})$. $\square$

\end{proof}

Theorem 2.6 below is essentially a restatement of a result of \cite{You-Zhu}, here we also give a very simple proof of it, which considerably simplifies and improves the proof given in \cite{You-Zhu}.

\begin{theorem}
 Let $(X,\mathscr{U})$ be a complete Hausdorff  uniform space, $\leq$  a partial ordering
on $X$, and  $\phi: X\rightarrow \bar{L}^{0}(\mathcal{F})$
proper and  bounded from below by $\eta_{0}\in L^{0}(\mathcal{F})$.  Further, if the
following assumptions are  satisfied:\\

\noindent $(1)$ for each $x\in X$, $S(x)=\{y\in X :x\leq y\}$ is closed;\\
\noindent $(2)$ $x\leq y$ $\Rightarrow $ $\phi(y)\leq \phi(x)$;\\
\noindent $(3)$ $G$ is a totally ordered subset in $X$ such that $\{\phi(g),g \in G\}$ is a $d_{\varepsilon,\lambda}-$Cauchy net in
$L^{0}(\mathcal{F})$, then $\{x_{g},g\in G\} $ is a Cauchy net in $X$, where $x_{g}=g$ for any $g\in G$.

 \noindent Then for each
$x_{0}\in dom(\phi)$, there exists $\bar{x}\in dom(\phi)$ such that
$x_{0}\leq \bar{x}$ and $\bar{x}$ is a maximal element in $X$.
\end{theorem}

\begin{proof}
Given an arbitrary $x_{0}$ in $dom(\phi)$,  we only need to prove
that there exists a maximal element in the set $S(x_{0})$. For this,
by the Zorn's lemma we must prove that any totally ordered subset
$G$ in the $S(x_{0})$ has an upper bound in $S(x_{0})$.

It is clear that $\phi(G)$ has an upper bound $\phi(x_{0})$, but $\phi$ is bounded from below, and hence $\phi(G)$ is contained in $L^{0}(\mathcal{F})$. By Lemma 2.5, $\{\phi(g),g \in G\}$ is a $d_{\varepsilon,\lambda}-$Cauchy net in
$L^{0}(\mathcal{F})$, then $\{x_{g},g\in G\} $ is a Cauchy net in $X$ by (3), and hence convergent to some $\hat{x}$ in $X$. Further, $\hat{x}$ belongs to $S(x_{0})$ by (1).

We now prove that $\hat{x}$ is an upper bound of $G$. In fact, let $g_{0}$ be any element in $G$, since
 $\{x_{g},g\in G \textmd{ and } g\geq g_{0}\} $ is a cofinal subnet of $\{x_{g},g\in G\} $ and contained in $S(g_{0})$, $\hat{x}\geq g_{0}$ holds.

Finally, $S(x_{0})$ has a maximal element $\bar{x}$, it is clear that $\bar{x}$ is just desired. $\square$

\end{proof}

\begin{remark}

When the probability space $(\Omega,{\cal F},P)$ is trivial, namely ${\cal F}=\{\emptyset,\Omega\}$, an $\bar{L}^{0}(\mathcal{F})$-valued function reduces to be an extended real-valued function and You and Zhu \cite{You-Zhu} proved that Theorem 2.6 also implies Theorem 1 of Br$\phi$ndsted \cite{Brndsted-1974}, which can derive the Bishop-Phelps lemma \cite{B-R}, Ekeland's variational principle \cite{Ekeland-1974} and Caristi's fixed point theorem \cite{Caristi}. In the next section, we will use Theorem 2.6 to establish the Ekeland's variational principle  and Caristi's fixed point theorem on complete $RM$ spaces.

\end{remark}

\subsection{The Ekeland's variational principle on a $d_{\varepsilon,\lambda}$-complete $RM$ space}
 \label{}

 In all the  vector-valued extensions of the Ekeland's variational principle, it is key to properly define the lower semicontinuity for a vector-valued function \cite{Gopfert-Tammer-Za,Finet, Yousuke Araya}. Recently, we have found  that a kind of lower semicontinuity for $\bar{L}^{0}-$valued functions is very suitable for the study of conditional risk measures \cite{Guotx-recent},  Definition 2.8 below is a direct extension of the lower semicontinuity to an $RM $ space.

\begin{definition}
  Let $(E,d)$ be a random
metric space  with base $(\Omega,{\cal F},P)$. A function
$f:E\rightarrow \bar{L}^{0}(\mathcal {F})$ is called {\it{${\cal
T}_{\varepsilon,\lambda}-$lower semicontinuous}} if  $epi(f)$ is closed in $(E,\mathcal{
T}_{\varepsilon,\lambda})\times (L^{0}(\mathcal {F}),\mathcal{
T}_{\varepsilon,\lambda})$. Similarly, a function
$f:E\rightarrow \bar{L}^{0}(\mathcal {F})$ is called {\it{${\cal
T}_{c}-$lower semicontinuous}} if  $epi(f)$ is
closed in $(E,\mathcal{ T}_{c})\times (L^{0}(\mathcal {F}),\mathcal{ T}_{c})$.

\end{definition}

\begin{remark}

The ${\cal T}_{c}-$lower semicontinuity in Definition 2.8 was given in \cite{Filipovic-Kupper} for an $\bar{L}^{0}$-valued function defined on $RN$ modules. In \cite{You-Zhu,Bai-Xiong}, a function $f$ from an $RM$ space $(E,d)$ to $\bar{L}^{0}(\mathcal {F})$ is called lower semicontinuous at $x$ if there exists a subsequence $\{x_{n_{k}}, k\in N\}$ for any sequence $\{x_{n}, n\in N\}$ convergent to $x$ in the ${\cal
T}_{\varepsilon,\lambda}$ such that $f(x)\leq \underline{\lim}_{k} f(x_{n_{k}})$. Obviously, this kind of lower semicontinuity is stronger than the ${\cal T}_{\varepsilon,\lambda}-$lower semicontinuity, and it seems that the latter is more natural.

\end{remark}

Theorem 2.10 below is the Ekeland's variational principle  on $d_{\varepsilon,\lambda}-$complete random metric spaces.

\begin{theorem}

 Let $(E,d)$ be a $d_{\varepsilon,\lambda}-$complete random metric space  with base
$(\Omega,{\cal F},P)$ and  $\phi: E\rightarrow
\bar{L}^{0}(\mathcal{F})$  a proper  $\mathcal
{T}_{\varepsilon,\lambda}-$lower semicontinuous function which is
bounded from below. Then for each $x_{0}\in dom(\phi)$, there exists
$v\in dom(\phi)$ such that the following are satisfied:\\

\noindent$(1)$ $\phi(v)\leq \phi(x_{0})-d(x_{0},v)$;\\
\noindent$(2)$ for each  $x\neq v$ in $E$,  $\phi(x)\nleq \phi(v)-d(x,v)$ holds, namely there exists $ A_{x} \in \mathcal
{F}$ with $P(A_{x})>0$ such that $\phi(x)>\phi(v)-d(x,v)$ on $A_{x}$.

\end{theorem}

\begin{proof}
Define an ordering $\leq $ on $E$ as follows: $x\leq y$ if and only
if either $x=y$,  or $x$ and $y \in dom(\phi) \textmd{ are  such that } d(x,y) \leq
\phi(x)- \phi(y) $.

It is easy to check that $\leq $ is a partial ordering. We now prove
that $X=E$ and $\phi$ satisfy the hypotheses  of Theorem 2.6 as
follows.

(1). Given an arbitrary  $x$ in $E$, then we now prove $S(x):=\{y\in
E :x\leq y\}$ is $\mathcal {T}_{\varepsilon,\lambda}$-closed. In
fact, first $S(x)=\{x\}$ when $x$ does not belong to $dom(\phi)$,
then when $x$ belongs to $dom(\phi)$, $S(x)=\{y\in dom(\phi): d(x,y)
\leq \phi(x)- \phi(y)\}$, we will prove, at this time, $S(x)=\{y\in
dom(\phi):  d(x,y) \leq \phi(x)- \phi(y)\}$ is $\mathcal
{T}_{\varepsilon,\lambda}-$closd as follows. Since
$(\varepsilon,\lambda)-$topology is metrizable,  let us suppose that
a sequence $\{x_{n}:n \in N\} $  in $ S(x)$ converges in the
$(\varepsilon,\lambda)-$topology to $a$, then  $ d(x,x_{n})\leq
\phi(x)- \phi(x_{n}), \forall n\in N$. Let $r_{n}=
\phi(x)-d(x,x_{n})$, then we have $(x_{n},r_{n}) \in
epi(\phi),\forall n\in N$, further since $\phi$ is $\mathcal
{T}_{\varepsilon,\lambda}-$lower semicontinuous, one can have that
$epi(\phi)$ is closed  in
$(E,\mathcal{T}_{\varepsilon,\lambda})\times (L^{0}(\mathcal
{F}),\mathcal{T}_{\varepsilon,\lambda})$ by definition and since
$\{x_{n},n \in N\}$ converges in the
$(\varepsilon,\lambda)-$topology to $a$, it follows that
$\{r_{n}:n\in N\}$ converges in the $(\varepsilon,\lambda)-$topology
to $\phi(x)-d(x,a)$. Thus one can obtain $(a,\phi(x)-d(x,a))\in
epi(\phi)$, namely $S(x)$ is $\mathcal
{T}_{\varepsilon,\lambda}-$closd.

 (2). By the definition of the ordering $\leq$ on $E$,  it is obvious that  $x\leq y \Rightarrow \phi(y)\leq
\phi(x)$.

(3). Suppose that  $M$ is a totally ordered subset in $E$ such that
$\{\phi(m),m\in M\}$ is a $d_{\varepsilon,\lambda}-$Cauchy net in $L^{0}(\mathcal{F})$, where $M$ is still understood
as the net $\{x_{m},m\in M \}$, where $x_{m}=m$. By the definition of the ordering $\leq$ on $E$, $M=\{x_{m},m \in M\}$ is a
$d_{\varepsilon,\lambda}-$Cauchy net in $E$.

Thus according to  Theorem 2.6,  for each $x_{0}\in dom(\phi)$, there
exists $v\in dom(\phi)$ such that $x_{0}\leq v$ and $v$ is a
maximal element in $E$, which just satisfies our desire. $\square$

\end{proof}

Theorem 2.10  can be easily derived from Theorem 2.11 below by replacing
$E$ with its closed subset $M:=\{x\in E: \phi(x)\leq
\phi(x_{0})-d(x_{0},x) \}$ and taking $\phi|_{M}$ instead of $\phi$, in fact, one can easily see that they are equivalent to each other.

\begin{theorem}
 Let $(E,d)$ be a $d_{\varepsilon,\lambda}-$complete random metric space  with base
$(\Omega,{\cal F},P)$ and  $\phi: E\rightarrow
\bar{L}^{0}(\mathcal{F})$  a proper  $\mathcal
{T}_{\varepsilon,\lambda}-$lower semicontinuous function which is
bounded from below. Then there exists $v\in E$ such that  $\phi(x)\nleq \phi(v)-d(x,v),\forall x\neq v$.

\end{theorem}

 Theorem 2.12 below is  the  Caristi's fixed point theorem on $d_{\varepsilon,\lambda}$-complete random
metric spaces. One can prove Theorem 2.12 by Theorem 2.11 and that they are equivalent to each other.

\begin{theorem}
 Let $(E,d)$ be a $d_{\varepsilon,\lambda}-$complete
random metric space  with base $(\Omega,{\cal F},P)$,  $\phi:
E\rightarrow \bar{L}^{0}(\mathcal{F})$  a proper  $\mathcal
{T}_{\varepsilon,\lambda}-$lower semicontinuous function which is
bounded from below, and $T:E \rightarrow E$  a mapping such that $\phi(Tu)+d(Tu,u)\leq\phi(u),\forall u\in E.$ Then $T$ has a fixed
point.

\end{theorem}

\begin{remark}

Since we employ a weaker and more natural lower semicontinuity than that used in the papers \cite{You-Zhu,Bai-Xiong} and  also allow the function to take values in $\bar{L}^{0}(\mathcal{F})$ unlike the papers \cite{You-Zhu,Bai-Xiong} where only $L^{0}(\mathcal{F})$-valued  functions were considered, our Theorems 2.10, 2.11 and 2.12 improve those in \cite{You-Zhu,Bai-Xiong}.

\end{remark}


\section{The precise forms of the Ekeland's variational principle on a complete $RN$ module under
 two kinds of topologies}
\label{}

 The Ekeland's variational principle for a proper and  lower bounded  extended real-valued function $f$ on a complete metric space $E$ can give the location of the approximate minimal point of $f$, since the following fact always holds: for any given positive real number $\varepsilon$, there exists a point $x_{\varepsilon}$ in $E$ such that $f(x_{\varepsilon})\leq inf f(E)+\varepsilon$. Whereas such a simple fact unnecessarily holds for a proper and  lower bounded $\bar{L}^{0}(\mathcal{F})$-valued function $f$ on a $d_{\varepsilon,\lambda}-$complete $RM$ space, which makes our Theorem 2.10 not able to give the location of the approximate minimal point $v$ of $f$. The weakness of Theorem 2.10 can be overcome in the context of complete $RN$ modules through Theorem 3.5 below. On the other hand, since the $\mathcal {T}_{c}$-topology on $L^{0}(\mathcal{F})$ is too strong to ensure that an a.s. convergent sequence is necessarily convergent in the $\mathcal{T}_{c}$-topology, Theorems 2.10, 2.11 and 2.12 derived from Theorem 2.6 do not have the corresponding version when an $RM$ space is endowed with the $d_{c}$-uniformity, such an unpleasant state of affairs can also be overcome by making use of the relations between the basic results derived from the two kinds of topologies \cite{Guotx-relation}. To sum up, the results obtained under the framework of $RN$ modules overcome all the above shortcomings and thus are also most useful in Section 4 and in the future optimization problems for conditional risk measures.

This section  is devoted to establishing the precise form of Ekeland's variational principle
for lower semicontinuous $\bar{L}^{0}-$valued functions on complete $RN$ modules
 under  two kinds of topologies (namely
$\mathcal {T}_{\varepsilon,\lambda}$ and $\mathcal {T}_{c}$),
namely Theorem 3.6 and 3.10 below.

\begin{definition} [\cite{Guotx-basic}]

 An ordered pair
$(E,\|\cdot\|)$ is called a random normed space (briefly, an $RN$ space) over $K$ with base
$(\Omega,{\cal F},P)$ if $E$ is a linear space and $\|\cdot\|$ is a mapping from $E$ to
$L^{0}_{+}(\mathcal{F})$ such that the following three axioms are
satisfied:\\

\noindent(1) $\|x\|=0$ if and only if $x=\theta$ (the null vector of $E$);\\
\noindent(2) $\|\alpha x\|=|\alpha |\|x\|,\forall \alpha\in K$ and $x\in E$;\\
\noindent(3) $\|x+y\|\leq \|x\|+\|y\|,\forall x,y\in E$,\\

\noindent where the mapping $\|\cdot\|$ is called the random norm
on $E$ and $\|x\|$ is called the random norm of a vector $x\in E$.

In addition, if $E$ is left module over the algebra $L^{0}({\cal F},K)$ such that the following is also satisfied:\\

\noindent(4) $\|\xi x\|=|\xi|\|x\|,\forall\xi\in L^{0}({\cal F},K)$ and $x\in E$,

\noindent then such an $RN$ space is called an $RN$ module over $K$ with base $(\Omega, \mathcal {F},P)$ and such a random norm $\|\cdot\|$ is called an $L^{0}$-norm on $E$.
\end{definition}

 Let $(E,\|\cdot\|)$ be an $RN$ space over $K$ with base $(\Omega, \mathcal {F},P)$, then $E$ is an $RM$ space endowed with the random metric $d: E\times E \rightarrow L^{0}_{+}(\mathcal {F})$ by $d(x,y)=\|x-y\|, \forall x,y\in E$. Throughout this paper, the $(\varepsilon,\lambda)$-topology and $\mathcal {T}_{c}$-topology are always assumed to be those induced by the random metric $d$. Since every $RN$ space uniquely determines a probabilistic normed space (briefly, a $PN$ space) \cite{SS}, in this sense an $RN$ space can be regarded as a special $PN$ space, so the $(\varepsilon,\lambda)$-topology is a metrizable linear topology, please refer to \cite{CBA3,BJC1,BC,CS} for the studies related to the  $(\varepsilon,\lambda)$-topology for a general $PN$ space. In particular,  it is well known from \cite{Guotx-relation} that $(L^{0}(\mathcal {F},K),\mathcal {T}_{\varepsilon,\lambda})$ is a topological algebra over $K$ and  an $RN$ module $(E,\|\cdot\|)$ over $K$ with base $(\Omega, \mathcal {F},P)$ is a topological module over the topological algebra $(L^{0}(\mathcal {F},K),\mathcal {T}_{\varepsilon,\lambda})$ when $E$ is endowed with its $(\varepsilon,\lambda)$-topology. On the other hand, the  $\mathcal {T}_{c}$-topology for an $RN$ module is just the locally $L^{0}$-convex topology, in particular, $(L^{0}(\mathcal {F},K),\mathcal {T}_{c})$ is only a topological ring  and  an $RN$ module $(E,\|\cdot\|)$ over $K$ with base $(\Omega, \mathcal {F},P)$ is a topological module over the topological ring $(L^{0}(\mathcal {F},K),\mathcal {T}_{c})$ when $E$ is endowed with its locally $L^{0}$-convex topology, see \cite{Filipovic-Kupper} for details.

 Let $(E,\|\cdot\|)$ be an $RN$ module over $K$ with base $(\Omega, \mathcal {F},P)$,  $p_{A}=\tilde{I}_{A}\cdot p$
is called the $A-$stratification of $p$ for each given $A\in \mathcal {F}$ and $p$ in $E$. The so-called stratification structure of $E$ means that $E$ includes every stratification of an element in $E$.  Clearly, $p_{A}=\theta$ when
$P(A)=0$ and $p_{A}=p$ when $P(\Omega \setminus A)=0$, which  are both called trivial stratifications of $p$. Further, when $(\Omega,\mathcal {F},P)$ is trivial probability space every element in $E$ has merely the two trivial stratifications since $\mathcal{F}=\{\Omega,\emptyset\}$; when $(\Omega, \mathcal {F},P)$ is arbitrary, every element in $E$ can possess arbitrarily many nontrivial intermediate stratifications. It is this kind of rich stratification structure of $RN$ modules that makes the theory of $RN$ modules deeply developed and also become the most useful part of  random metric theory.

To introduce the main results of this paper, let us first recall:

\begin{definition} [\cite{Guotx-relation}]

 Let $E$ be a
left module over the algebra $L^0(\mathcal {F}, K)$. A formal sum
$\sum_{n \in N}\widetilde{I}_{A_n}x_n$ is called a \emph{{countable
concatenation}} of a sequence $\{x_n\mid n\in N\}$ in $E$ with
respect to a countable partition $\{A_n\mid n\in N\}$  of $\Omega$
to $\mathcal {F}$. Moreover,  a countable concatenation $\sum_{n \in
N}\widetilde{I}_{A_n}x_n$ is well defined or $\sum_{n \in
N}\widetilde{I}_{A_n}x_n\in E$ if there is $x\in E$ such that
$\widetilde{I}_{A_n}x=\widetilde{I}_{A_n}x_n, \forall n\in N$.  A
subset $G$ of $E$ is said to
 \emph {have the countable concatenation property} if
every countable concatenation $\sum_{n \in N}\tilde{I}_{A_n}x_n$ with $x_n\in
 G$ for each $n\in N$ still belongs to $G$, namely $\sum_{n  \in N}\tilde{I}_{A_n}x_n$ is well defined and there exists $x\in G$
 such that $x=\sum_{n \in N}\tilde{I}_{A_n}x_n$.
\end{definition}

\begin{definition}[\cite{Filipovic-Kupper}]
 Let $E$ be a left module over the algebra
$L^{0}(\mathcal {F})$ and  $f$ a function from $E$ to $\bar {L}^{0}(\mathcal{F})$, then \\

\noindent(1) $f$ is $L^{0}(\mathcal {F})$-convex if $f(\xi x+(1-\xi)y)\leq \xi
f(x)+(1-\xi)f(y)$ for all $x$ and $y$ in $E$ and $\xi \in L_{+}^{0}(\mathcal{F})$
such that $0\leq \xi \leq 1$ (Here we make the convention that
$0\cdot (\pm \infty)=0$ and $\infty -\infty =\infty~!)$.\\
\noindent(2) $f$ is said to have the local property if
$\tilde{I}_Af(x)=\tilde{I}_Af(\tilde{I}_{A}x)$ for all $x\in E$ and
$A\in \cal F$.
\end{definition}

It is well known from \cite{Filipovic-Kupper} that
 $f:E \rightarrow \bar{L}^{0}(\mathcal {F})$ is $L^{0}(\mathcal {F})$-convex iff $f$ has the local property and
$epi(f)$ is $L^{0}(\mathcal {F})$-convex.

\begin{lemma}

 Let $E$ be an $RN$ module over $R$ with base
$(\Omega,{\cal F},P)$, $G\subset E$ a subset such that $ \tilde{I}_{A} G+\tilde{I}_{A^{c}} G \subset G$ and $f:E\rightarrow
\bar{L}^{0}(\mathcal{F})$ a function  with the local property. Then
$\{f(x):x\in G\}$ is both directed downwards and directed upwards.

\end{lemma}

\begin{proof}
Let  $x$ and $y$ be any two elements in $G$ and $f^{0}(x)$ and
$f^{0}(y)$ arbitrarily chosen representatives of $f(x)$ and
$f(y)$, respectively.

Take $A=\{\omega \in \Omega:f^{0}(x)(\omega)\leq f^{0}(y)(\omega)\}$ and $z_{1}=\tilde{I}_{A} \cdot
x+\tilde{I}_{A^{c}} \cdot y$, then  $z_{1} \in G$ and it is easy
to check that $f(z_{1})=f(x)\bigwedge f(y)$ by the local property of
$f$ and hence $\{f(x):x\in G\}$ is directed downwards. Similarly, one can prove that $\{f(x):x\in G\}$ is  directed upwards. $\square$

\end{proof}

Let $(E,\|\cdot\|)$ be an $RN$ module over $R$
with base $(\Omega,{\cal F},P)$ and $G$ a subset of $E$. Since $G$
is an $RM$ space, as a subspace of the $RM$ space $(E,\|\cdot\|)$,
then we can say that $f:G \rightarrow \bar{L}^{0}(\mathcal{F})$ is
proper, $\mathcal{T}_{\varepsilon,\lambda}-$lower semicontinuous and
$\mathcal{T}_{c}-$lower semicontinuous in the sense of Section 2.

\begin{theorem}

Let $(E,\|\cdot\|)$ be an $RN$ module over $R$ with base
$(\Omega,{\cal F},P)$, $G\subset E$ a subset with   the countable
concatenation property and   $f:E\rightarrow
\bar{L}^{0}(\mathcal{F})$ have the local property. If $f|_{G}$ is
proper and bounded from below on $G$ $($resp.,  bounded from above on
$G$$)$. Then for each $\varepsilon \in L^{0}_{++}(\mathcal {F})$,
there exists $x_{\varepsilon} \in G$ such that
$f(x_{\varepsilon})\leq \bigwedge f(G)+\varepsilon$ (accordingly,
$f(x_{\varepsilon})\geq \bigvee f(G)-\varepsilon$).
\end{theorem}

\begin{proof}
We only need to prove the case when $f$ is bounded from below on $G$ as
follows.

Since $G$ has the countable concatenation property, $G$ must satisfy the property that $ \tilde{I}_{A} G+\tilde{I}_{A^{c}} G \subset G$. Further, since $f$ has the local property, $\{f(x):x\in
G\}$ is directed downwards by Lemma 3.4. According to Proposition
2.1, there exists a sequence $\{x_{n},n\in N\}$ in $G$ such that
$\{f(x_{n}),n\in N\}$  converges to $ \eta:=\bigwedge f(G)$ in a nonincreasing way, then  it follows from Egoroff's Theorem that $\{f(x_{n}),n\in N\}$ converges $P-$uniformly to $\eta$. Thus there exists $E_{m}\in \mathcal {F}$ for
each $m\in N$ such that
$P(\Omega\setminus E_{m})< \frac{1}{m} $ and $\{f(x_{n}),n\in N\}$
converges uniformly to $\eta$ on $E_{m}$, which is denoted by
$f(x_{n})\rightrightarrows \eta $  on $E_{m}$ for convenience.

Since $P(\bigcup^{\infty}_{n=1}E_{n})=1$, we can suppose $\Omega=\bigcup^{\infty}_{n=1}E_{n}$. Further, let $E_{n}'=\bigcup_{k=1}^{n} E_{k}, \forall n\in N$, then $\bigcup^{\infty}_{m=1}E_{m}'=\bigcup^{\infty}_{m=1}E_{m}=\Omega$ and $E_{m}'\subset E_{m+1}', \forall m\in N$.

Taking $F_{1}=E_{1}', F_{n}=E_{n}'\setminus \bigcup^{n-1}_{k=1}E_{k}', \forall n\geq 2$, one
can have $F_{i}\bigcap  F_{j}=\emptyset ( i\neq j)$ and $\bigcup^{\infty}_{n=1}F_{n}=\Omega$.

First, we prove that for each $k\in N$ there exists $x^{(k)}\in
G$ such that $f(x^{(k)})\leq \eta + \frac{1}{k}$ as follows. Let
$f^{0}(x_{n})$ and $\eta ^{0}$ be arbitrarily chosen
representatives of $f(x_{n})$ and $\eta$, respectively. From
$f(x_{n})\rightrightarrows \eta $  on $F_{m},\forall m\in N$ , it
follows that for each $k \in N $, there exists $N(k,m)\in N$
such that $|f^{0}(x_{n})(\omega)-\eta ^{0}(\omega)|\leq \frac{1}{k},\forall
\omega \in F_{m}  \mbox{ and }  n\geq N(k,m)$, and hence $f(x_{n})\leq
\eta +\frac{1}{k}$ on $F_{m}$,$\forall n\geq N(k,m)$.

 By the hypothesis that $G$ has the countable concatenation property, one can
 have $x^{(k)}:=\sum ^{\infty}_{m=1}\tilde{I}_{F_{m}}\cdot x_{N(k,m)}\in G$ is well defined and
 $\tilde{I}_{F_{m}}\cdot x^{(k)}=\tilde{I}_{F_{m}}\cdot x_{N(k,m)},\forall  m\in N$. Hence one can have
 $\tilde{I}_{F_{m}}\cdot f(\tilde{I}_{F_{m}}\cdot x^{(k)})=\tilde{I}_{F_{m}}\cdot f(\tilde{I}_{F_{m}}\cdot x_{N(k,m)})$, which
 implies $\tilde{I}_{F_{m}}\cdot f(x^{(k)})=\tilde{I}_{F_{m}}\cdot  f(x_{N(k,m)}) \leq \tilde{I}_{F_{m}}\cdot (\eta +\frac{1}{k}),\forall  m\in N$ by
  the local property of $f$. Since $\bigcup^{\infty}_{n=1}F_{n}=\Omega$, we have  $f(x^{(k)})\leq \eta +\frac{1}{k}$.

 Second, for each $\varepsilon \in  L^{0}_{++}(\mathcal {F})$, let $A_{1}=\{\omega :\varepsilon ^{0}(\omega)\geq 1\}$,  $A_{k+1}=\{\omega:\frac{1}{k+1} \leq \varepsilon ^{0} (\omega)<\frac{1}{k}\},\forall k\geq 1 $, where $\varepsilon ^{0}$ is an
 arbitrarily chosen representative of $\varepsilon$. Then  $\{ A_{i}, i\geq 1 \}$ forms a countable partition of $\Omega$ to $\mathcal {F}$. It
 is easy to see that $f(x^{(k)})\leq \eta +\frac{1}{k}\leq \eta +\varepsilon$ on $A_{k}$, $\forall k\geq 1 $.

From the countable concatenation property of $G$, it follows that
$x_{\varepsilon}:=\sum ^{\infty}_{k=1}\tilde{I}_{A_{k}}\cdot x^{(k)} \in G$ is
well defined. Further, by the local property of $f$, it is obvious
that $\tilde{I}_{A_{k}}f(x_{\varepsilon})=\tilde{I}_{A_{k}}f(\tilde{I}_{A_{k}}
\cdot x_{\varepsilon}) =\tilde{I}_{A_{k}}f(\tilde{I}_{A_{k}} \cdot
x^{(k)})=\tilde{I}_{A_{k}}f( x^{(k)}) \leq \tilde{I}_{A_{k}} \cdot
(\eta +\varepsilon)$, $\forall k\geq 1$. Since
$\bigcup^{\infty}_{k=1}A_{k}=\Omega$, we have $f(x_{\varepsilon})\leq \eta
+\varepsilon.$

Similarly, we can prove this theorem when  $f$ is bounded from above
on $G$. $\square$

\end{proof}

 According to Theorem 3.5, there does exist $x_{0}$ satisfying the
 hypothesis of Theorem 3.6 below if $\varphi$ has the local property and $G$ has the countable
concatenation property. By Theorem 2.10, one can obtain the following precise form of Ekeland's variational principle on a $\mathcal{T}_{\varepsilon,\lambda}-$complete RN module:

\begin{theorem}
 Let $(E,\|\cdot\|)$ be a $\mathcal
{T}_{\varepsilon,\lambda}-$complete RN module over $R$ with base
$(\Omega,{\cal F},P)$, $G$ a $\mathcal
{T}_{\varepsilon,\lambda}-$closed subset of $E$, $ \varepsilon \in
L^{0}_{++}(\mathcal {F})$ and  $\varphi: G \rightarrow
\bar{L}^{0}(\mathcal{F})$ a proper,
$\mathcal{T}_{\varepsilon,\lambda}-$lower semicontinuous and bounded
from below on $G$. Then for each point $x_{0}\in G$ satisfying
$\varphi(x_{0})\leq \bigwedge \varphi(G)+\varepsilon$ and each $
\alpha \in L^{0}_{++}(\mathcal {F})$, there exists $z\in G$ such
that the following are satisfied:\\

\noindent $(1)$ $\varphi(z)  \leq \varphi(x_{0})-\alpha \|z-x_{0}\| $;\\
\noindent $(2)$ $\|z-x_{0}\| \leq \alpha^{-1}\cdot \varepsilon$;\\
\noindent $(3)$ for each $x \in G$ such that $x\neq z$,  $\varphi(x)\nleq
\varphi(z)-\alpha\|x-z\|$.

\end{theorem}

To obtain the precise form of Ekeland's variational principle under the  locally $L^{0}$-convex topology, we need the following key results obtained in \cite{Guotx-relation,Guotx-recent}:

\begin{proposition}  [\cite{Guotx-relation}] Let
$(E,\|\cdot\|)$ be an $RN$ module over $K$ with base $(\Omega,{\cal
F},P)$. Then  $E$ is ${\cal
T}_{\varepsilon,\lambda}-$complete if and only if $E$ is ${\cal
T}_{c}-$complete and has the countable concatenation property.

\end{proposition}

Obviously, Proposition 7.2.3 \cite{Guotx-recent} also holds for a subset with the countable concatenation property:

\begin{proposition}
Let $(E,\|\cdot\|)$ be an $RN$ module
 over $R$ with base $(\Omega,{\cal F},P)$ such that $E$ has
the countable concatenation property, $G\subset E$ a subset with the countable concatenation property and $f:E\rightarrow
\bar{L}^{0}(\mathcal {F})$ a function with the local property.
Then $f|_{G}$ is ${\cal T}_{\varepsilon,\lambda}-$lower semiconinuous  iff
$f|_{G}$ is ${\cal T}_{c}-$lower semicontinuous, in particular,  this is true
when $f$ is  $L^{0}(\mathcal {F})$-convex.

\end{proposition}

\begin{proposition}[\cite{Guotx-recent}]

Let $(E,\|\cdot\|)$ be an RN module  over $K$ with base
$(\Omega,{\cal F},P)$   and $A$ a subset with the countable
concatenation property of $E$. Then
$\bar{A}_c=\bar{A}_{\varepsilon,\lambda}$, where $\bar{A}_c$ and
$\bar{A}_{\varepsilon,\lambda}$ stand for the ${\cal T}_c-$closure and
${\cal T}_{\varepsilon,\lambda}-$closure of $A$, respectively.

\end{proposition}

We can now give the precise form of Ekeland's variational principle under $\mathcal{T}_{c}$, namely Theorem 3.10 below, the difference between Theorem 3.6 and Theorem 3.10 lies in that the local property of $\varphi$ in Theorem 3.10 must be assumed to apply Proposition 3.8.

\begin{theorem}
 Let $(E,\|\cdot\|)$ be a $\mathcal
{T}_{c}-$complete RN module over $R$ with base $(\Omega,{\cal F},P)$
such that $E$ has the countable concatenation property, $ \varepsilon \in L^{0}_{++}(\mathcal {F})$ and  $\varphi: E \rightarrow \bar{L}^{0}(\mathcal{F})$ have the local property. If $G \subset E$ is a $\mathcal {T}_{c}-$closed subset with  the countable concatenation property and  $\varphi|_{G}$ is a proper, $\mathcal {T}_{c}-$lower semicontinuous and bounded from below on $G$, then for each point $x_{0}\in G$ satisfying $\varphi(x_{0}) \leq
\bigwedge \varphi(G)+\varepsilon$ and each $ \alpha \in
L^{0}_{++}(\mathcal {F})$, there exists $z\in G$
such that the following are satisfied:\\

\noindent $(1)$ $\varphi(z)  \leq \varphi(x_{0})-\alpha \|z-x_{0}\| $;\\
\noindent $(2)$ $\|z-x_{0}\| \leq \alpha^{-1}\cdot \varepsilon$;\\
\noindent $(3)$ for each $x \in G$ such that $x\neq z$,  $\varphi(x)\nleq \varphi(z)-\alpha\|x-z\|$.

\end{theorem}

\begin{proof}

Since $E$ is  ${\cal T}_{\varepsilon,\lambda}-$complete by Proposition 3.7, $\varphi$ is also  ${\cal T}_{\varepsilon,\lambda}-$lower semicontinuous on $G$ by Proposition 3.8 and $G$ is  $\mathcal {T}_{\varepsilon,\lambda}-$closed by Proposition 3.9, then our desired conclusion follows immediately from Theorem 3.6.        $\square$

\end{proof}

Similarly, we can obtain the following Caristi's fixed point theorem under $\mathcal
{T}_{c}$:

\begin{theorem}
 Let $(E,\|\cdot\|)$ be a $\mathcal
{T}_{c}-$complete RN module over $R$ with base $(\Omega,{\cal F},P)$
such that $E$ has  the countable concatenation property and  $\varphi: E \rightarrow \bar{L}^{0}(\mathcal {F})$  a proper
function  such that $\varphi$ is $\mathcal {T}_{c}-$lower semicontinuous and
bounded from below and has the local property.  If $T:E \rightarrow E$ is a mapping  such that
$\varphi(Tu) +\|Tu-u\| \leq\varphi(u),\forall u\in E$, then $T$ has
a fixed point.

\end{theorem}


\section{ The Bishop-Phelps theorem in complete  RN modules}
\label{}

 In this section, applying the results in Section 3  we establish the Bishop-Phelps theorems in complete  $RN$ modules under the framework of random conjugate spaces and proceed under the two kinds of topologies, respectively. The main results in this section are Theorems 4.2 and  4.3 below. To introduce them, we first give some necessary notation and terminology.

Let us first recall the notion of a random conjugate space, though it can be introduced for any $RN$ space \cite{Guotx-basic}, to save space we only need the following:

\begin{definition}[\cite{Guotx-relation}]
 Let $(E,\|\cdot\|)$ be an $RN$ module  over
$K$ with base ($\Omega,{\cal F},P$). Then
$E^{\ast}_{\varepsilon,\lambda}=\{f:E\rightarrow L^{0}({\cal
F},K)| f$ is a continuous module homomorphism from $(E,{\cal
T}_{\varepsilon,\lambda})$ to $(L^{0}({\cal F},K),{\cal
T}_{\varepsilon,\lambda})\}$ and $E^{\ast}_{c}=\{f:E\rightarrow
L^{0}({\cal F},K)| f$ is a continuous module homomorphism from
$(E,{\cal T}_{c})$ to $(L^{0}({\cal F},K),{\cal T}_{c})\}$, are
called the random conjugate spaces of $(E,\|\cdot\|)$ under ${\cal
T}_{\varepsilon,\lambda}$ and ${\cal T}_{c}$, respectively.

\end{definition}

 It is well known from \cite{Guotx-relation} that an $RN$ module $(E,\|\cdot\|)$ over
$K$ with base $(\Omega,{\cal F},P)$ has the same random conjugate space under ${\cal
T}_{\varepsilon,\lambda}$ and ${\cal T}_{c}$, namely $E^{\ast}_{\varepsilon,\lambda}=E^{\ast}_{c}$, and thus they can be denoted by the same notation  $E^{\ast}$. It is well known that a function $f$ from $E$ to $L^{0}({\cal F},K)$ belongs to $E^{\ast}$ if and only if $f$ is a linear operator and there is $\xi\in L^{0}_+(\mathcal {F})$ such that $|f(x)|\leqslant\xi\cdot\|x\|,\forall x\in E$, so an element of  $E^{\ast}$ is also called an a.s. bounded random linear functional on $E$. Further, define $\|\cdot\|^{\ast}:E^{\ast}\rightarrow L^{0}_+(\mathcal {F})$ by
$\|f\|^{\ast}=\wedge\{\xi\in
L^{0}_+(\mathcal {F}):|f(x)|\leqslant\xi\cdot\|x\|,\forall x\in E\}$, then
$(E^{\ast},\|\cdot\|^{\ast})$ is also an $RN$ module over $K$ with base
$(\Omega,{\cal F},P)$ and $\|f\|^{\ast}=\vee\{|f(x)|: x\in
E$ and $\|x\|\leqslant 1\}$ for any $f \in E^{\ast}$. Besides, it is known from \cite{Guotx-recent} that $E^{\ast}$ is ${\cal T}_{\varepsilon,\lambda}-$complete, so $E^{\ast}$ must have the countable concatenation property \cite{Guotx-relation}.

 Let $E$ be a left module over the
algebra $L^{0}(\mathcal{F},K)$, a nonempty subset $M$ of $E$ is
called $L^{0}(\mathcal{F})$-convex  if $\xi x+\eta y \in M$ for any $x$ and $y \in M$ and $\xi $ and $\eta\in L^{0}_{+}(\mathcal {F})$ such that $\xi+\eta=1$. In addition, it is
called an $L^{0}(\mathcal{F})$-convex cone if  $\xi x+\eta y \in M$ for any $x$ and $y \in M$ and $\xi $ and $\eta\in L^{0}_{+}(\mathcal {F})$, further $M$ is called pointed if $M\bigcap (-M)=\theta$.

Let $E$ be an $RN$ module over $R$ with base $(\Omega,{\cal F},P)$, $G\subset E$ a subset and $f\in E^{\ast}\setminus \{0\}$ such that $f$ is bounded from above on $G$. If $x \in G$ is such that $f(x)=\bigvee f(G)$, then $x$ is called a support point of $f$ and $f$ is called an a.s. bounded random linear functional supporting $G$ at $x$.

We can now state the main results in this section.

\begin{theorem}
 Let $(E,\|\cdot\|)$ be a $\mathcal
{T}_{c}-$complete RN module over $R$ with base $(\Omega,{\cal F},P)$
such that $E$ has the countable concatenation
property and $G$  a  $\mathcal {T}_{c}-$closed $L^{0}(\mathcal
{F})-$convex subset of $E$ such that $G$ has the countable
concatenation property. Then the set of  support points of $G$ is
$\mathcal {T}_{c}-$dense in the $\mathcal {T}_{c}-$boundary of $G$ $($briefly, $\partial_{c}G$$)$.

\end{theorem}

A ${\cal T}_{\varepsilon,\lambda}$-complete $L^{0}(\mathcal{F})$-convex subset $G$ must have the countable concatenation property, but we wonder whether Theorem 4.2 is true or not under the $(\varepsilon,\lambda)$-topology, namely, let $(E,\|\cdot\|)$ be a $\mathcal{T}_{\varepsilon,\lambda}-$complete RN module over $R$ with base $(\Omega,{\cal F},P)$
 and $G$  a  $\mathcal {T}_{\varepsilon,\lambda}-$closed $L^{0}(\mathcal
{F})-$convex subset of $E$, then is the set of  support points of $G$
$\mathcal {T}_{\varepsilon,\lambda}-$dense in the $\mathcal {T}_{\varepsilon,\lambda}-$boundary of $G$ $($briefly, $\partial_{\varepsilon,\lambda}G$$)$?

\begin{theorem}
 Let $(E,\|\cdot\|)$ be a $\mathcal
{T}_{c}-$complete RN module over $R$ with base $(\Omega,{\cal F},P)$
such that $E$ has the countable concatenation property, and $G$ an
a.s. bounded $($namely, $\bigvee\{\|p\|:p\in G\} \in
L^{0}_{+}(\mathcal{F})$$)$,  $\mathcal {T}_{c}$-closed and
$L^{0}(\mathcal {F})$-convex subset of $E$ such that $G$ has the
countable concatenation property.  Then the set of  a.s. bounded
random linear functionals supporting $G$ is $\mathcal {T}_{c}-$dense
in $E^{\ast}$.

\end{theorem}

 Theorem 4.3 still holds under the $(\varepsilon,\lambda)$-topology, namely we also have:  let $(E,\|\cdot\|)$ be a $\mathcal
{T}_{\varepsilon,\lambda}-$complete RN module over $R$ with base $(\Omega,{\cal F},P)$ and $G$ an
a.s. bounded,  $\mathcal {T}_{\varepsilon,\lambda}$-closed and
$L^{0}(\mathcal {F})$-convex subset of $E$, then the set of  a.s. bounded
random linear functionals supporting $G$ is $\mathcal {T}_{\varepsilon,\lambda}-$dense
in $E^{\ast}$, see the paragraph below Corollary 4.16 for details.

To prove the two theorems, we need a series of preparations.
Propositions 4.6 and 4.7 below are the hyperplane separation
theorems in $RN$ modules under the locally $L^{0}-$convex topology,
which play an important role in this section. To introduce them,
we first give Definition 4.4 as well as Proposition 4.5
below, which were given by Guo in \cite{Guotx-relation}.

\begin{definition} [\cite{Guotx-relation}]
  Let $E$ be
an $L^{0}({\cal F},K)-$module  and $G$ a subset of $E$. The set of
countable concatenations $\sum_{n\geqslant 1}\tilde{I}_{A_n}x_n$
with $x_n\in
 G$ for each $n\in N$ is called {\it the countable concatenation
 hull} of $G$, denoted by $H_{cc}(G)$.
\end{definition}

 Clearly, we have $H_{cc}(G)\supset G$ for any subset $G$ of an  $L^{0}({\cal
 F},K)-$module $E$, and $G$ has the countable concatenation property
 iff $H_{cc}(G)=G$.

\begin{proposition}[\cite{Guotx-relation}]
 Let $E$ be a left module over the
 algebra $L^{0}({\cal F},K)$, $M$ and $G$ any two nonempty subsets
 of $E$ such that $\tilde{I}_{A}M+\tilde{I}_{A^{c}}M\subset M$ and $\tilde{I}_{A}G+\tilde{I}_{A^{c}}G\subset G$. If $H_{cc}(M)\,\cap\,H_{cc}(G)=\emptyset$, then there exists an
 ${\cal F}-$measurable subset $H(M,G)$ unique a.s. such that the
 following are satisfied:\\

 \noindent(1) $P(H(M,G))>0$;\\
 \noindent(2) $\tilde{I}_{A}M\,\cap\,\tilde{I}_{A}G=\emptyset$ for all $A\in
 {\cal F},A\subset H(M,G)$ with $P(A)>0$;\\
\noindent (3) $\tilde{I}_{A}M\,\cap\,\tilde{I}_{A}G\neq\emptyset$ for all $A\in
 {\cal F},A\subset \Omega\backslash H(M,G)$ with $P(A)>0$.

\end{proposition}

 Let $E$, $M$
and $G$ be the same as in Proposition 4.5 such that $H_{cc}(M)\,\cap\,H_{cc}(G)=\emptyset$, then $H(M,G)$ is called the hereditarily
 disjoint stratification of $H$ and $M$, and $P(H(M,G))$ is called the
 hereditarily disjoint probability of $H$ and $G$.

Propositions 4.6 and 4.7 below are merely the special case of  the corresponding
theorems of \cite{Guotx-relation} and \cite{Guotx-shiguang} which were originally given for general random locally convex modules.

\begin{proposition}  [\cite{Guotx-relation}] Let $(E,\|\cdot\|)$ be an $RN$
 module over $K$ with base
 $(\Omega,{\cal F},P)$, $x\in E$ and $G$ a nonempty
 ${\cal T}_{c}-$closed $L^{0}(\mathcal {F})-$convex subset of $E$ such that
 $x\notin G$ and $G$ has the countable concatenation property. Then
 there exists an $f\in E^{*}$ such that  $$(Ref)(x) \geq \vee\{(Ref)(y)\,|\,y\in G\}$$
 \noindent and
$$(Ref)(x)>\vee\{(Ref)(y)\,|\,y\in G\} \mbox{
 on } H(\{x\},G), $$
where $(Ref)(x)=Re(f(x)),\forall x\in E$.

\end{proposition}

\begin{proposition}  [\cite{Guotx-shiguang}]

 Let
$(E,\|\cdot\|)$ be an $RN$ module over $K$ with base
 $(\Omega,{\cal F},P)$ and $G$ and $M$
 two nonempty $L^{0}(\mathcal {F})-$convex
 subsets of $E$ such that the ${\cal T}_{c}-$interior $G^{o}$ of $G$ is not empty and $H_{cc}(G^{o})\bigcap H_{cc}(M)=\emptyset$.
 Then there exists $f\in E^{*}$ such
 that $$(Ref)(x)\leq (Ref)(y) \textmd{ for all } x\in G \mbox{ and } y\in M$$
 \noindent and
 $$(Ref)(x)<(Ref)(y) \textmd{ on } H(G^{o},M) \textmd{ for all } x\in G^{o} \mbox{ and } y\in M. $$

\end{proposition}

\begin{definition}

 Let $E$ be an $RN$ module over $R$ with
base $(\Omega,{\cal F},P)$, $f\in
E^{\ast}$ and $k \in L^{0}_{++}(\mathcal {F})$.
Define
    $$K(f,k)=\{y\in E:k\|y\|\leq f(y)\}.$$
It is easy to see that $K(f,k)$ is a pointed, closed  and $L^{0}(\mathcal{F})$-convex cone under each of $\mathcal
{T}_{\varepsilon,\lambda}$ and $\mathcal{T}_{c}$.

\end{definition}

\begin{lemma}

 Let $(E,\|\cdot\|)$ be a $\mathcal
{T}_{c}-$complete RN module over $R$ with base
$(\Omega,{\cal F},P)$ such that $E$ has
 the countable concatenation property, $k
\in L^{0}_{++}(\mathcal {F})$,  $G\subset E$ a
$\mathcal {T}_{c}-$closed subset with
 the countable concatenation property. Further, if  $f\in
E^{\ast}$  is bounded from above on $G$, and $\varepsilon \in
L^{0}_{++}(\mathcal {F})$ and $z\in G$ are such that $\bigvee f(G)\leq
f(z)+\varepsilon$,  then  there exists $x_{0}\in G$ such that:\\

\noindent$(1)$ $x_{0}\in K(f,k)+z$;\\
\noindent$(2)$ $\|x_{0}-z\|\leq k^{-1}\cdot \varepsilon$;\\
\noindent$(3)$ $G\bigcap (K(f,k)+x_{0})=\{x_{0}\}$.

\end{lemma}

\begin{proof} Applying $\varphi=-f$ and $\alpha =k$ to Theorem 3.10,
then  there exists
$x_{0}\in G$ such that the following are satisfied:\\
\noindent(1') $k\cdot \|x_{0}-z\|\leq f(x_{0})-f(z)$;\\
\noindent(2') $\|x_{0}-z\|\leq k^{-1}\cdot \varepsilon$;\\
\noindent(3') for each  $x \in G $ such that $x\neq x_{0}$,  $k\cdot \|x-x_{0}\| \nleq
f(x)-f(x_{0})$ holds.\\
\indent Obviously, (1'), (2') and (3') amount to our desired conclusions.  $\square$

\end{proof}

To prove the key Lemma 4.12, we need  Lemma 4.10, which is very easy and thus whose proof is omitted,  and  Proposition 4.11  below.

\begin{lemma}

 Let $(E,\|\cdot\|)$ be an RN module over
$R$ with base $(\Omega,{\cal F},P)$ such that $E$ has
 the countable concatenation property and $f: E
\rightarrow \bar{L}^{0}(\mathcal{F})$ a function with the local
property. Then $epi(f)$ has the countable concatenation property.

\end{lemma}

\begin{proposition}[\cite{Guotx-shiguang}]

 Let $(E,\|\cdot\|)$ be an $RN$
 module over $K$ with base
 $(\Omega,{\cal F},P)$. If a subset $G$ of $E$ has the countable
 concatenation property, then so does the  ${\cal T}_{c}-$interior $G^{o}$ of
 $G$.
\end{proposition}

\begin{lemma}
Let $(E,\|\cdot\|)$ be an $RN$ module  over $R$ with base
$(\Omega,{\cal F},P)$ such that $E$ has the countable concatenation
property, $k \in L^{0}_{++}(\mathcal {F})$,  $f\in E^{\ast}$ and  $G$ an
$L^{0}(\mathcal {F})-$convex subset of $E$ such that $G$  has the countable
concatenation property. Further, if $x_{0}\in G$ satisfies $G\bigcap
(K(f,k)+x_{0})=\{x_{0}\}$, then there exists $g \in E^{\ast}$ such
that
$$\bigvee g(G)=g(x_{0}) \textmd{ and } \|f-g\|^{\ast}\leq k.$$

\end{lemma}

\begin{proof}
 Define a function  $\phi: E\rightarrow
L^{0}(\mathcal{F})$ by $\phi(x)=k\|x\|-f(x), \forall x\in E$. It is easy to check that $\phi$ is $L^{0}(\mathcal {F})-$convex and has the local
property.

Let $C_{1}:=epi(\phi)$ and $C_{2}:= (G-x_{0})\times \{0\}$.

We now prove that $C_{1}$ and $C_{2}$ satisfy the hypotheses of Proposition 4.7 as follows.

(1).  Obviously,  $C_{1}$ and $C_{2}$ are nonempty by $(0,0)\in C_{1}\bigcap C_{2}$. Since $\phi$ and $C$ are both $L^{0}(\mathcal {F})-$convex,  it is easy to check that $C_{1}$ and $C_{2}$ are both
 $L^{0}(\mathcal {F})-$convex.

It is clear that the ${\cal T}_{c}-$interior of $C_{1}$ denoted by  $C_{1}^{o}=\{(x,r)\in E\times
L^{0}(\mathcal{F}): \phi(x)< r \textmd{ on } \Omega\}$ is not empty by  $(0,1)\in
C_{1}^{o}.$

(2). Since $E$ has the countable concatenation property and $\phi$ has
the local property, it follows that $C_{1}$ has the countable
concatenation property by Lemma 4.10. Thus  $C_{1}^{o}$ has the
countable concatenation property by Proposition 4.11. By the countable concatenation
property of $G$, it is easy to check that $C_{2}$ has the countable concatenation property.

(3). We can now prove  $\tilde{I}_{A}\cdot C_{1}^{o}\bigcap
\tilde{I}_{A}\cdot C_{2}=\emptyset$  for any $A \in \mathcal {F}$ with
$P(A)>0$ as follows.

First, from  $G\bigcap (K(f,k)+x_{0})=\{x_{0}\}$, one can have
$(G-x_{0})\bigcap K(f,k)=\{0\}$, which implies $C_{1}\bigcap
C_{2}=\{(0,0)\}$, and it is clear that  $C_{1}^{o}\bigcap
C_{2}=\emptyset$ since $(0,0) \in \partial_{c}C_{1}$.

Second, from $C_{1}\bigcap C_{2}=\{(0,0)\}$, we can deduce $\tilde{I}_{A}\cdot
C_{1}\bigcap \tilde{I}_{A}\cdot C_{2}=\tilde{I}_{A}\cdot \{(0,0)\}$
for any $ A \in \mathcal {F}$ with $P(A)>0$. Otherwise, there exists
some $B\in \mathcal {F}$ with $P(B)>0$ and $\hat{y}\in E\times
L^{0}(\mathcal{F})$ such that  $\tilde{I}_{B}\cdot \hat{y} \in
\tilde{I}_{B}\cdot C_{1}\bigcap \tilde{I}_{B}\cdot C_{2}$ and
$\tilde{I}_{B}\cdot \hat{y}\neq \tilde{I}_{B}\cdot (0,0).$ Let us
take $z=\tilde{I}_{B}\cdot \hat{y} +\tilde{I}_{B^{c}}\cdot (0,0)$, then it is easy to  see that $\tilde{I}_{B^{c}}\cdot (0,0)\in\tilde{I}_{B^{c}}\cdot (C_{1}\bigcap C_{2}) \subset
\tilde{I}_{B^{c}}\cdot C_{1} \bigcap \tilde{I}_{B^{c}}\cdot C_{2}.$
Thus we can have $z\in C_{1}\bigcap C_{2}=\{(0,0)\}$, which implies $ \tilde{I}_{B}\cdot \hat{y}=
\tilde{I}_{B}\cdot (0,0)$,  a contradiction.

Third, we consider the problem in the relative topology. Since $\tilde{I}_{A}\cdot C_{1}^{o}$ is the relative $\mathcal {T}_{c}$-interior of $\tilde{I}_{A}\cdot C_{1}$ in $\tilde{I}_{A}\cdot (E\times L^{0}(\mathcal{F}))$ and $\tilde{I}_{A}\cdot (0,0)$ is a relative $\mathcal {T}_{c}$-boundary point of $\tilde{I}_{A}\cdot C_{1}$ in $\tilde{I}_{A}\cdot (E\times L^{0}(\mathcal{F}))$, we can have $\tilde{I}_{A}\cdot C_{1}^{o}\bigcap
\tilde{I}_{A}\cdot C_{2}=\emptyset$ for any $ A \in \mathcal {F}$ with
$P(A)>0.$

Since $(E\times L^{0}(\mathcal{F}))^{\ast}=E^{\ast} \times L^{0}(\mathcal{F})^{\ast}=E^{\ast} \times L^{0}(\mathcal{F})$ by noting $L^{0}(\mathcal{F})^{\ast}=L^{0}(\mathcal{F})$, then applying Proposition 4.7 to the special case that $H( C_{1}^{o}, C_{2})=\Omega$ we have that there exists $F\in E^{\ast}\times
L^{0}(\mathcal{F})$ such that

$$F(p)<F(q)\textmd{ on } \Omega \textmd{  for all } p\in C_{2} \textmd{ and } q\in C_{1}^{o} \eqno(4.1)$$

\noindent and

$$F(p)\leq F(q)  \textmd{  for all } p\in C_{2} \textmd{ and } q\in C_{1}, \eqno(4.2)$$

\noindent and hence we have $\bigvee F(C_{2})=0=\bigwedge  F(C_{1}).$

Further, there exists $g\in E^{\ast}$
and $r^{\ast} \in L^{0}(\mathcal{F})$ such that $F(x,r)=g(x)+r^{\ast}\cdot r, \forall (x,r)\in E\times L^{0}(\mathcal{F}).$
From  $(0,1)\in C_{1}^{o}$, it follows that $F(0,1)>0$ on $\Omega$ by (4.1), which implies $r^{\ast}>0$ on $\Omega$. Thus we can, without loss of generality, suppose $r^{\ast}=1$, and hence $F(x,r)=g(x)+ r,\forall (x,r)\in E\times L^{0}(\mathcal{F}).$

 Since $(x-x_{0},0 ) \in  C_{2}$ for any $x\in G$,  it follows that $0\geq F(x-x_{0},0)=g(x)-g(x_{0})$  by (4.2) and hence  $g(x_{0})=\bigvee  g(G).$

Since  $(x,\phi(x))\in C_{1}$ for $x\in E$, one can have that $0\leq F(x,\phi(x))=g(x)+\phi(x)=g(x)+ k\|x\|-f(x)$ by (4.2), which implies $\|f-g\|^{\ast}\leq k.$ $\square$

 \end{proof}

\begin{corollary}
 Let $(E,\|\cdot\|)$ be a $\mathcal
{T}_{c}-$complete RN module over $R$ with base $(\Omega,{\cal F},P)$
such that $E$ has  the countable concatenation
property, $G$ a $\mathcal {T}_{c}-$closed $L^{0}(\mathcal {F})-$convex subset of $E$ such that $G$
has the countable concatenation property, and  $f\in E^{\ast}$ which is bounded from above on $G$. Further, if  $\varepsilon \in
L^{0}_{++}(\mathcal {F})$ and $z\in G$ are such that $\bigvee f(G)\leq
f(z)+\varepsilon$, then for each $k \in L^{0}_{++}(\mathcal {F})$, there exist $g \in E^{\ast}$ and $x_{0}\in G$ such that:\\

\noindent $(1)$ $g(x_{0})=\bigvee g(G)$;\\
\noindent $(2)$ $\|x_{0}-z\|\leq k^{-1}\cdot \varepsilon$;\\
\noindent $(3)$ $\|f-g\|^{\ast}\leq k.$

\end{corollary}

\begin{proof}

By Lemma 4.9, there exists $x_{0}\in G$ such that
$$x_{0}\in K(f,k)+z, \|x_{0}-z\|\leq k^{-1}\cdot \varepsilon  \textmd{ and } G\bigcap
(K(f,k)+x_{0})=\{x_{0}\}.$$

Thus by Lemma 4.12,  there exists $g \in
E^{\ast}$ such that $g(x_{0})=\bigvee g(G)$ and $\|f-g\|^{\ast}\leq
k$.  $\square$

\end{proof}

We can now prove Theorem 4.2:

 \noindent{\bf Proof of Theorem 4.2.}\quad   We can, without loss of generality, suppose $\partial_{c} G\neq \emptyset$. Let  $z$ be in  $\partial_{c} G$  and $\delta $ in
$L^{0}_{++}(\mathcal {F})$,  then there exists some $y\in E\backslash
G$ such that $\|y-z\|\leq \frac{\delta}{2}$.

Since $y\bar{\in} G$, there exists $f \in E^{\ast}\backslash \{0\}$,  we can, without loss of generality, suppose that
$\|f\|^{\ast}=\tilde{I}_{[\|f\|^{\ast}\neq 0]}$ (otherwise we can consider $(\|f\|^{\ast})^{-1}\cdot f$) such that

 $$ \bigvee f(G)< f(y) \textmd{ on } H(\{y\}, G)\textmd{ and } \bigvee f(G)\leq f(y)  $$

\noindent by Proposition 4.6.

From $f(y)\leq f(z)+\|y-z\|\leq f(z) +\frac{\delta}{2}$, we have $\bigvee
f(G)\leq f(y) \leq  f(z)+\frac{\delta}{2}$.

 Then taking $\varepsilon =\frac{\delta}{2}$ and $k=\frac{1}{2}$ in
Corollary 4.13, it follows that
 there exists $g \in E^{\ast}$ and $x_{0}\in G$ such
that
$$g(x_{0})=\bigvee g(G), \|x_{0}-z\|\leq \delta \textmd{ and }\| f-g\|^{\ast}\leq
\frac{1}{2}. \eqno(4.3) $$

 We now prove $g\neq 0$. Since $f\neq 0$, we can have $P([\|f\|^{\ast}\neq 0]) >0$. Furthermore, since $\|f\|^{\ast}=\tilde{I}_{[\|f\|^{\ast}\neq 0]}$ and $\| f-g\|^{\ast}\leq
\frac{1}{2}$, one can  have $g\neq 0$.

 Thus it is clear that  $x_{0}$ is just desired from (4.3). $\square$

\begin{corollary}
 Let $(E,\|\cdot\|)$ be a $\mathcal
{T}_{c}-$complete RN module over $R$ with base $(\Omega,{\cal F},P)$
such that $E$ has the countable concatenation
property,  $G$ a $\mathcal {T}_{c}-$closed $L^{0}(\mathcal
{F})-$convex  subset of $E$ such that $G$ has the countable
concatenation property, and $f\in E^{\ast}\backslash \{0\}$  which is bounded from above on $G$. Then for any $\delta
\in L^{0}_{++}(\mathcal {F})$ with $\delta< \|f\|^{\ast}$ on
$[\|f\|^{\ast}>0]$, there exists  $g \in E^{\ast}\backslash \{0\}$  supporting $G$ such that
$\|f-g\|^{\ast}\leq \delta$.

\end{corollary}

\begin{proof}
 We can choose $z\in G$ such that $\bigvee f(G)\leq f(z)+1$ by Theorem 3.5. Taking
$\varepsilon=1 \textmd{ and } k=\delta$ in Corollary 4.13, then
there exists $g \in E^{\ast}$ and $x_{0}\in G$ such that

$$ \|f-g\|^{\ast}\leq \delta<
\|f\|^{\ast}\textmd{ on }[\|f\|^{\ast}>0]\textmd{ and } g(x_{0})=\bigvee g(G). $$

 Thus  $g$ is desired.   $\square$

\end{proof}

We can now prove Theorem 4.3:

\noindent{\bf Proof of Theorem 4.3.}\quad
 Since $G$ is a.s. bounded, it is easy to see that $f\in E^{\ast}$ is  bounded from above  on $G$, then we can get the conclusion from Corollary 4.14. $\square$

\begin{definition}[\cite{Z-G}]
 An $RN$ module $E$ is called $\mathcal
{T}_{\varepsilon,\lambda}$ (resp.,  $\mathcal
{T}_{c}$)-random subreflexive  if  the set of all $f \in E^{\ast}$ satisfying
$f(x)=\|f\|^{\ast}$ for some $x\in E$ with $\|x\|\leq 1$,
is $\mathcal{T}_{\varepsilon,\lambda}$(accordingly, $\mathcal
{T}_{c}$)-dense in $E^{\ast}$.

 \end{definition}

From Theorem 4.3, one can obtain Corollary 4.16 below:

\begin{corollary} [\cite{Z-G}]

 Let $(E,\|\cdot\|)$ be a $\mathcal
{T}_{c}-$complete RN module over $R$ with base $(\Omega,{\cal F},P)$
such that $E$ has  the countable concatenation
property. Then  $E$ is $\mathcal {T}_{c}-$random  subreflexive.

\end{corollary}

In \cite{Z-G}, Zhao and Guo illustrate that Corollary 4.16 may not hold if $E$ does not have the countable concatenation
property. In addition, according to  Propositions 3.7, 3.9 and Corollary 4.16,  we can
obtain  Corollary 4.17 below (namely, the $\mathcal{T}_{\varepsilon,\lambda}-$random  subreflexivity) by  the countable
concatenation property of the set of all $f \in E^{\ast}$ satisfying
$f(x)=\|f\|^{\ast}$ for some $x\in E$ with $\|x\|\leq 1$. More generally, by Propositions 3.7 and 3.9 and by the observation that $H_{G}\bigcup \{0\}$ has the countable concatenation property, where $H_{G}$ denotes the set of  a.s. bounded random linear functionals supporting $G$, one can similarly see that Theorem 4.3 still holds under the $(\varepsilon,\lambda)$-topology, see the paragraph following the statement of Theorem 4.3.

\begin{corollary}[\cite{Z-G}]

Let $(E,\|\cdot\|)$ be a $\mathcal
{T}_{\varepsilon,\lambda}-$complete RN module over $R$ with base $(\Omega,{\cal F},P)$. Then  $E$ is $\mathcal
{T}_{\varepsilon,\lambda}-$random  subreflexive.

\end{corollary}

\begin{remark}

The proofs given in \cite{Z-G} of Corollaries 4.16 and 4.17 are constructive and thus skillful so that Zhao and Guo can avoid the transfinite induction method, whereas our proofs here are relatively simple but we have to  employ the Ekeland's variational principle, namely we inexplicitly  use the transfinite induction method.

\end{remark}








\end{document}